\documentclass{amsart}
\usepackage{amssymb,setspace}
\usepackage[hmarginratio=1:1]{geometry}
\usepackage{ifpdf}
\ifpdf
 \usepackage[hyperindex,pagebackref]{hyperref}%
\else
 \expandafter\ifx\csname dvipdfm\endcsname\relax
 \usepackage[hypertex,hyperindex,pagebackref]{hyperref}
 \else
 \usepackage[dvipdfm,hyperindex,pagebackref]{hyperref}
 \fi
\fi
\allowdisplaybreaks[4]
\numberwithin{equation}{section}
\theoremstyle{plain}
\newtheorem{thm}{Theorem}[section]
\newtheorem{cor}{Corollary}[thm]
\newtheorem{lem}{Lemma}[section]
\theoremstyle{definition}
\newtheorem{dfn}{Definition}[section]
\theoremstyle{remark}
\newtheorem{rem}{Remark}[section]
\DeclareMathOperator{\td}{d\mspace{-1mu}}

\begin{document}

\title[Inequalities of Hermite-Hadamard type and applications]
{Inequalities of Hermite-Hadamard type for extended $\boldsymbol{s}$-convex functions and applications to means}

\author[B.-Y. Xi]{Bo-Yan Xi}
\address[B.-Y. Xi]{College of Mathematics, Inner Mongolia University for Nationalities, Tongliao City, Inner Mongolia Autonomous Region, 028043, China}
\email{\href{mailto: B.-Y. Xi <baoyintu78@qq.com>}{baoyintu78@qq.com}, \href{mailto: B.-Y. Xi <baoyintu68@sohu.com>}{baoyintu68@sohu.com}}

\author[F. Qi]{Feng Qi}
\address[F. Qi]{Department of Mathematics, College of Science, Tianjin Polytechnic University, Tianjin City, 300387, China}
\email{\href{mailto: F. Qi <qifeng618@gmail.com>}{qifeng618@gmail.com}, \href{mailto: F. Qi <qifeng618@hotmail.com>}{qifeng618@hotmail.com}, \href{mailto: F. Qi <qifeng618@qq.com>}{qifeng618@qq.com}}
\urladdr{\url{http://qifeng618.wordpress.com}}

\begin{abstract}
In the paper, the authors introduce a new concept ``extended $s$-convex functions'', establish some new integral inequalities of Hermite-Hadamard type for this kind of functions, and apply these inequalities to derive some inequalities of special means.
\end{abstract}



\date{Submitted on 23 May 2012; accepted on 18 November 2013}

\thanks{Please cite this article as ``Bo-Yan Xi and Feng Qi, \textit{Inequalities of Hermite-Hadamard type for extended $s$-convex functions and applications to means}, Journal of Nonlinear and Convex Analysis (2014), in press.''}

\maketitle

\section{Introduction}

Throughout this paper, we use the following notation:
\begin{equation}
\begin{aligned}
\mathbb{R}&=(-\infty,\infty), & \mathbb{R}_0 & =[0,\infty),& & \text{and}& \mathbb{R}_+ &=(0,\infty).
\end{aligned}
\end{equation}
\par
The following definitions are well known in the literature.

\begin{dfn}\label{Xi-Qi-dfn1}\label{convex-dfn}
A function $f:I\subseteq\mathbb{R}\to\mathbb{R}$ is said to be convex if
\begin{equation}
f(\lambda x+(1-\lambda)y)\le\lambda f(x)+(1-\lambda)f(y)
\end{equation}
holds for all $x,y\in I$ and $\lambda\in[0,1].$
\end{dfn}

\begin{dfn}[\cite{Dragomir-Pecaric-Persson-SJM-335-341}]\label{P-convex-dfn}
A function $f:I\subseteq\mathbb{R}\to\mathbb{R}_{0}$ is said to be $P$-convex if
\begin{equation}
f(\lambda x+(1-\lambda)y)\le f(x)+f(y)
\end{equation}
holds for all $x,y\in I$ and $\lambda\in[0,1]$.
\end{dfn}

\begin{dfn}[\cite{Godunova-Levin-SJM-138-142}]\label{Godunova-Levin-convex-dfn}
A function $f:I\subseteq\mathbb{R}\to\mathbb{R}_0$ is said to be a Godunova-Levin function if $f$ is nonnegative and
\begin{equation}
f(\lambda x+(1-\lambda)y)\le \frac{f(x)}{\lambda}+\frac{f(y)}{1-\lambda}
\end{equation}
holds for all $x,y\in I$ and $\lambda\in(0,1)$.
\end{dfn}

\begin{dfn}[\cite{Hudzik-Maligranda-AM-100-111}]\label{s-convex-dfn}
Let $s\in(0, 1]$ be a real number. A function $f:\mathbb{R}_0\to\mathbb{R}_0$ is said to be $s$-convex (in the second sense) if
\begin{equation}
f(\lambda x+(1-\lambda)y)\le \lambda^sf(x)+(1-\lambda)^sf(y)
\end{equation}
holds for all $x,y\in I$ and $\lambda\in[0,1].$
\end{dfn}

In recent decades, a lot of inequalities of Hermite-Hadamard type for various kinds of convex functions have been established. Some of them may be recited as follows.

\begin{thm}[{\cite{Dragomir-Agarwal-AML-98-95}}]\label{Dragomir-Agarwal-AML-98-95-thm}
Let $f:I^\circ\subseteq\mathbb{R}\to\mathbb{R}$ be a differentiable mapping on $I^\circ$ and $a,b\in I^\circ$ with $a<b$. If $| f'(x)|$ is convex on $[a,b]$, then
\begin{equation}\label{eq-xi-1.5}
\biggl\lvert\frac{f(a)+f(b)}2-\frac1{b-a}\int_a^bf(x)\td x\biggr|\le\frac{(b-a)\bigl(|f'(a)| +|f'(b)|\bigr)}8.
\end{equation}
\end{thm}

\begin{thm}[{\cite{Kirmaci--Bakula-Ozdemir-Pecaric-AMC-26-35}}]
Let $f:I\subseteq\mathbb{R}_0\to\mathbb{R}$ be differentiable on $I^\circ$ and $a,b\in I$ with $a<b$.
If $|f'(x)|^{q}$ is $s$-convex on $[a,b]$ for some fixed $s\in(0, 1]$ and $q\ge 1$, then
\begin{multline}\label{eq-xi-1.9}
\biggl\lvert\frac{f(a)+f(b)}2-\frac1{b-a}\int_a^bf(x)\td x\biggr| \\
\le\frac{b-a}{2}\biggl(\frac1{2}\biggr)^{1-1/q}\biggl[\frac{2+1/2^s}{(s+1)(s+2)}\biggr]^{1/q} \bigl[|f'(a)|^{q}+|f'(b)| ^{q}\bigr]^{1/q}.
\end{multline}
\end{thm}

\begin{thm}[{\cite{Hussain-Bhatti-Iqbal-PUJM-51-60}}]
Let $f:I\subseteq\mathbb{R}_0\to\mathbb{R}$ be differentiable on $I^\circ$, $a,b\in I$ with $a<b$, and $f'\in L[a,b]$. If $|f'(x)|^q$ is $s$-convex on $[a,b]$ for some fixed $s\in(0, 1]$ and $q>1$, then
\begin{multline}\label{eq-xi-1.11}
\biggl|f\biggl(\frac{a+b}2\biggr) -\frac1{b-a}\int_a^bf(x)\td x\biggr|
\le\frac{b-a}{4}\biggl[\frac1{(s+1)(s+2)}\biggr]^{1/q}\biggl(\frac1{2}\biggr)^{1/p} \\
\times\biggl\{\biggl[|f'(a)|^q+(s+1)\biggl|f'\biggl(\frac{a+b}{2}\biggr)\biggr|^q\biggr]^{1/q} +\biggl[|f'(b)|^q +(s+1)\biggl|f'\biggl(\frac{a+b}{2}\biggr)\biggr|^q\biggr]^{1/q}\biggr\},
\end{multline}
where $\frac1p+\frac1q=1$.
\end{thm}

\begin{thm}[{\cite{Sarikaya-Set-Ozdemir-CMA-2191-2199}}] \label{Sarikaya-Set-Ozdemir-CMA-2191-2199-thm}
Let $f:I\subseteq\mathbb{R}_0\to\mathbb{R}$ be differentiable on $I^\circ$, $a,b\in I$ with $a<b$, and $f'\in L[a,b]$.
If $|f'(x)|$ is $s$-convex on $[a,b]$ for some $s\in(0, 1]$, then
\begin{multline}\label{sep-eq-1.12}
\biggl|\frac16\biggl[{f(a)+f(b)}+4f\biggl(\frac{a+b}2\biggr)\biggr] -\frac1{b-a}\int_a^bf(x)\td x\biggr|\\*
\le\frac{(s-4)6^{s+1}+2\times5^{s+2}-2\times3^{s+2}+2}{6^{s+2}(s+1)(s+2)}(b-a)\bigl(|f'(a)|+|f'(b)|\bigr),
\end{multline}
where $\frac{1}{p}+\frac{1}{q}=1$.
\end{thm}

Some inequalities of Hermite-Hadamard type were also obtained in~\cite{Hadramard-Convex-Xi-Filomat.tex, H-H-Bai-Wang-Qi-2012.tex, chun-ling-Hermite.tex, difference-hermite-hadamard.tex, GMJ-2013-062.tex, H-H-Shuang-Ye-Munich.tex, Wang-Ineq-H-H-type-Analysis.tex, Xi-Bai-Qi-Hadamard-2011-AEQMath.tex, Hadramard-Convex-Xi-August-2011.tex, Hadramard-Convex-Xi-September-2011.tex, GA-H-H-Zhang-lematema.tex} and related references therein.
\par
In this paper, we will introduce a new concept ``extended $s$-convex functions'', establish some new integral inequalities of Hermite-Hadamard type for extended $s$-convex functions, and apply these newly established integral inequalities to derive some inequalities of special means. These results generalize inequalities stated in Theorems~\ref{Dragomir-Agarwal-AML-98-95-thm} to~\ref{Sarikaya-Set-Ozdemir-CMA-2191-2199-thm}.

\section{Definition and lemmas}

We first define the concept ``extended $s$-convex functions'' and establish an integral identity.

\begin{dfn}\label{Xi-Qi-dfn4}
For some $s\in[-1, 1]$, a function $f:I\subseteq\mathbb{R}_0\to\mathbb{R}_0$ is said to be extended $s$-convex if
\begin{equation}
f(\lambda x+(1-\lambda)y)\le \lambda^sf(x)+(1-\lambda)^sf(y)
\end{equation}
holds for all $x,y\in I$ and $\lambda\in(0,1).$
\end{dfn}

It is obvious that the extended $1$-convex function, $0$-convex function, and $-1$-convex function are just the usually convex function in Definition~\ref{convex-dfn}, the $P$-convex functions in Definition~\ref{P-convex-dfn}, and Godunova-Levin convex function in Definition~\ref{Godunova-Levin-convex-dfn}, respectively. It is also clear that Definition~\ref{Xi-Qi-dfn4} extends Definition~\ref{s-convex-dfn}.
\par
For establishing new integral inequalities of Hermite-Hadamard type for extended $s$-convex functions, we need the following integral identity.

\begin{lem}\label{lem1-September-2011-xi}
Let $f:I\subseteq\mathbb{R}\to\mathbb{R}$ be differentiable on $I^\circ$ and $a,b\in I$ with $a<b$. If $f'\in L[a,b]$ and $\lambda,\mu\in\mathbb{R}$, then
\begin{multline*}
\frac{\lambda f(a)+\mu f(b)}2+\frac{2-\lambda-\mu} 2f\biggl(\frac{a+b}2\biggr)-\frac1{b-a}\int_a^bf(x)\td x\\*
=\frac{b-a}4\int_0^1\biggl[(1-\lambda-t)f'\biggl(ta+(1-t)\frac{a+b}2\biggr) +(\mu-t)f'\biggl(t\frac{a+b}2+(1-t)b\biggr)\biggr]\td t.
\end{multline*}
\end{lem}

\begin{proof}
integrating by parts and changing variable of definite integral yield
\begin{align*}
&\quad\int_0^1(1-\lambda-t)f'\biggl(ta+(1-t)\frac{a+b}2\biggr)\td t\\
&=-\frac2{b-a}\biggl[(1-\lambda-t)f\biggl(ta+(1-t)\frac{a+b}2\biggr)\bigg\vert_0^1 +\int_0^1f\biggl(ta+(1-t)\frac{a+b}2\biggr)\td t\biggr] \\
&=\frac2{b-a}\biggl[\lambda f(a) +(1-\lambda)f\biggl(\frac{a+b}2\biggr)\biggr]
-\frac4{(b-a)^2}\int_a^{(a+b)/2}f(x)\td x
\end{align*}
and
\begin{align*}
&\quad\int_0^1(\mu-t)f'\biggl(t\frac{a+b}2+(1-t)b\biggr)\td t\\
&=-\frac2{b-a}\biggl[(\mu-t)f\biggl(t\frac{a+b}2+(1-t)b\biggr)\bigg\vert_0^1 +\int_0^1f\biggl(t\frac{a+b}2+(1-t)b\biggr)\td t\biggr] \\
&=\frac2{b-a}\biggl[(1-\mu)f\biggl(\frac{a+b}2\biggr)+\mu f(b)\biggr]
-\frac4{(b-a)^2}\int^b_{(a+b)/2}f(x)\td x.
\end{align*}
Adding these two equations leads to Lemma~\ref{lem1-September-2011-xi}.
\end{proof}

\begin{lem}\label{lem2-September-2011-xi}
Let $s>-1$, $0\le\xi\le 1$, $\omega\in \mathbb{R}\setminus\{0\}$, $\eta\ge 0$, and $\omega+\eta \ge0$. Then
\begin{multline}\label{lem2-Sep-2011-xi}
\int_0^1|\xi-t|(\omega t+\eta)^s\td t\\
=\frac{2(\omega\xi+\eta)^{s+2}-[\eta+(s+2)\omega\xi]\eta^{s+1} -[2\omega\xi+\eta+s\omega(\xi-1)-\omega](\omega+\eta)^{s+1}}{\omega^2(s+1)(s+2)}.
\end{multline}
In particular, if $(\omega, \eta)=(1, 0)$, $(1, 1)$, $(-1, 1)$, or $(-1, 2)$ respectively, then
\begin{align*}
\int_0^1|\xi-t|t^s\td t&=\frac{2\xi^{s+2}-(s+2)\xi+s+1}{(s+1)(s+2)},\\
\int_0^1|\xi-t|(1+t)^s \td t&=\frac{2(\xi+1)^{s+2} -[(s+2)\xi-s]2^{s+1}-(s+2)\xi-1}{(s+1)(s+2)},\\
\int_0^1|\xi-t|(1-t)^s \td t&=\frac{2(1-\xi)^{s+2}+(s+2)\xi-1}{(s+1)(s+2)},\\
\int_0^1|\xi-t|(2- t)^s\td t&=\frac{2(2-\xi)^{s+2} +[(s+2)\xi-2]2^{s+1}+(s+2)\eta-s-3}{(s+1)(s+2)}.
\end{align*}
\end{lem}

\begin{proof}
These follow from straightforward computation of definite integrals.
\end{proof}

\section{Some integral inequalities of Hermite-Hadamard type}

Now we are in a position to establish some new integral inequalities of Hermite-Hadamard type for differentiable extended $s$-convex functions.

\begin{thm}\label{thm1-September-2011-xi}
Let $f:I\subseteq\mathbb{R}_0\to\mathbb{R}$ be differentiable on $I^\circ$, $a,b\in I$ with $a<b$, $f'\in L[a,b]$, and $0\le\lambda,\mu\le 1$. If $|f'(x)|^q$ for $q\ge1$ is an extended $s$-convex function on $[a,b]$ for some fixed $s\in[-1, 1]$, then
\begin{enumerate}
\item
when $-1<s\le 1$, we have
\begin{align*}
&\quad\biggl\lvert\frac{\lambda f(a)+\mu f(b)}2+\frac{2-\lambda-\mu}2f\biggl(\frac{a+b}2\biggr) -\frac1{b-a}\int_a^bf(x)\td x\biggr|\\
&\le\frac{b-a}{2^{s/q+2}}\biggl[\frac{1}{(s+1)(s+2)}\biggr]^{1/q} \biggl\{\biggl(\frac{1}{2}-\lambda+\lambda^2\biggr)^{1-1/q}
\bigl[\bigl\{2(2-\lambda)^{s+2}\\
&\quad+[(s+2)\lambda-2]2^{s+1}+(s+2)\lambda-s-3\bigr\}|f'(a)|^q+\bigl\{2\lambda^{s+2} -(s+2)\lambda+s\\
&\quad+1\bigl\}|f'(b)|^{q}\bigr]^{1/q}+\biggr(\frac{1}{2}-\mu+\mu^2\biggl)^{1-1/q} \Bigl[\bigl(2\mu^{s+2}-(s+2)\mu+s+1\bigr)|f'(a)|^q \\
&\quad+\bigl\{2(2-\mu)^{s+2}+[(s+2)\mu-2]2^{s+1}+(s+2)\mu-s-3\bigr\}|f'(b)|^{q}\Bigr]^{1/q}\biggr\};
\end{align*}
\item
when $s=-1$, we have
\begin{multline}\label{thm1-2-September-2011-xi-ineq}
\biggl|f\biggl(\frac{a+b}2\biggr) -\frac1{b-a}\int_a^bf(x)\td x\biggr|\\
\le\frac{b-a}{2^{3-2/q}}\bigl\{\bigl[(2\ln 2-1)|f'(a)|^q
+|f'(b)|^{q}\bigr]^{1/q}+\bigl[|f'(a)|^q+(2\ln 2-1)|f'(b)|^{q}\bigr]^{1/q}\bigr\}.
\end{multline}
\end{enumerate}
\end{thm}

\begin{proof}
For $-1<s\le 1$, since $|f'(x)|^q$ is extended $s$-convex on $[a,b]$, by Lemmas~\ref{lem1-September-2011-xi} and~\ref{lem2-September-2011-xi} and by H\"older integral inequality, we have
\begin{align*}
&\quad\biggl|\frac{\lambda f(a)+\mu f(b)}2 +\frac{2-\lambda-\mu}2f\biggl(\frac{a+b}2\biggr) -\frac1{b-a}\int_a^bf(x)\td x\biggr| \\
&\le\frac{b-a}4\biggl[\int_0^1|1-\lambda-t|\biggl|f'\biggl(ta+(1-t)\frac{a+b}2\biggr)\biggr|\td t
+\int_0^1|\mu-t|\biggl|f'\biggl(t\frac{a+b}2
+(1-t)b\biggr)\biggr|\td t\biggr]\\
&=\frac{b-a}4\biggl[\int_0^1|1-\lambda-t|\biggl|f'\biggl(\frac{1+t}2 a+\frac{1-t}2 b\biggr)\biggr|\td t+\int_0^1|\mu-t|\biggl|f'\biggl(\frac{t}2 a +\frac{2-t}2 b\biggr)\biggr|\td t\biggr]\\
&\le\frac{b-a}{2^{s/q+2}}\biggl\{\biggl(\int_0^1|1-\lambda-t|\td t\biggr)^{1-1/q} \biggl[\int_0^1|1-\lambda-t|\bigl((1+t)^s|f'(a)|^q
+(1-t)^s|f'(b)|^q\bigr)\td t\biggr]^{1/q} \\
&\quad+\biggl(\int_0^1|\mu-t|\td t\biggr)^{1-1/q} \biggl[\int_0^1|\mu-t|\bigl(t^s|f'(a)|^q+(2-t)^s|f'(b)|^q\bigr)\td t\biggr]^{1/q}\biggr\}\\
&=\frac{b-a}{2^{s/q+2}}\biggl\{\biggl(\frac{1}{2}-\lambda+\lambda^2\biggr)^{1-1/q} \biggl[\frac{1}{(s+1)(s+2)}\bigl[\bigl(2(2-\lambda)^{s+2}+((s+2)\lambda-2)2^{s+1}\\
&\quad+(s+2)\lambda-s-3\bigr)|f'(a)|^q
+\bigl(2\lambda^{s+2}+s+1-(s+2)\lambda\bigr)|f'(b)|^{q}\bigr]\biggr]^{1/q} \\
&\quad+\biggl(\frac{1}{2}-\mu+\mu^2\biggr)^{1-1/q} \biggl[\frac{1}{(s+1)(s+2)}\bigl[\bigl(2\mu^{s+2}+s+1-(s+2)\mu\bigr)|f'(a)|^{q}\\
&\quad+\bigl(2(2-\mu)^{s+2}+((s+2)\mu-2)2^{s+1}+(s+2)\mu-s-3\bigr)|f'(a)|^q\bigr]\biggr]^{1/q}\biggr\}.
\end{align*}
\par
For $s=-1$, since $|f'(x)|^q$ is extended $-1$-convex on $[a,b]$, by Lemma~\ref{lem1-September-2011-xi} and H\"older integral inequality, we have
\begin{align*}
&\quad\biggl|f\biggl(\frac{a+b}2\biggr) -\frac1{b-a}\int_a^bf(x)\td x\biggr|\\
&\le\frac{b-a}4\biggl[\int_0^1\biggl|f'\biggl(ta+(1-t)\frac{a+b}2\biggr)\biggr|(1
-t)\td t+\int_0^1t\biggl|f'\biggl(t\frac{a+b}2+(1-t)b\biggr)\biggr|\td t\biggr]\\
&\le\frac{b-a}{2^{2-/q}}\biggl\{\biggl[\int_0^1(1-t)\td t\biggr]^{1-1/q}
\biggl[\int_0^1(1-t)\bigl((1+t)^{-1}|f'(a)|^q
+(1-t)^{-1}|f'(b)|^q\bigr)\td t\biggr]^{1/q}\\
&\quad+\biggl(\int_0^1t\td t\biggr)^{1-1/q} \biggl[\int_0^1t\bigl(t^{-1}|f'(a)|^q+(2-t)^{-1}|f'(b)|^q\bigr)\td t\biggr]^{1/q}\biggr\}\\
&=\frac{b-a}{2^{3-2/q}}\bigl\{\bigl[(2\ln 2-1)|f'(a)|^q+|f'(b)|^{q}\bigr]^{1/q} + \bigl[|f'(a)|^{q}+(2\ln 2-1)|f'(b)|^q\bigr]^{1/q}\bigr\}.
\end{align*}
Theorem~\ref{thm1-September-2011-xi} is proved.
\end{proof}

\begin{cor}\label{cor-3.1-1-2011-xi}
Under conditions of Theorem~\ref{thm1-September-2011-xi},
\begin{enumerate}
\item
if $q=1$ and $-1<s\le 1$, we have
\begin{equation}\label{sep-eq-3.9}
\begin{split}
&\quad\biggl\lvert\frac{\lambda f(a)+\mu f(b)}2+\frac{2-\lambda-\mu}2f\biggl(\frac{a+b}2\biggr) -\frac1{b-a}\int_a^bf(x)\td x\biggr|\\
&\le\frac{b-a}{2^{s+2}(s+1)(s+2)}\bigl\{\bigl[2(2-\lambda)^{s+2}+2\mu^{s+2} +((s+2)\lambda-2)2^{s+1}\\
&\quad+(s+2)(\lambda-\mu)-2]|f'(a)|+\bigl[2\lambda^{s+2}+2(2-\mu)^{s+2}\\
&\quad+((s+2)\mu-2)2^{s+1}+(s+2)(\mu-\lambda)-2]|f'(b)|\bigr\};
\end{split}
\end{equation}
\item
if $q=1$ and $s=-1$, we have
\begin{equation}\label{sep-eq-3.4}
\biggl|f\biggl(\frac{a+b}2\biggr) -\frac1{b-a}\int_a^bf(x)\td x\biggr|
\le(b-a)(\ln 2)\bigl[|f'(a)|+|f'(b)|\bigr].
\end{equation}
\end{enumerate}
\end{cor}

\begin{cor}\label{cor-3.1-2-2011-xi}
Under conditions of Theorem~\ref{thm1-September-2011-xi},
\begin{enumerate}
\item
when $\lambda=\mu$ and $-1<s\le 1$, we have
\begin{align*}
&\quad\biggl\lvert\lambda\frac{f(a)+f(b)}2+(1-\lambda)f\biggl(\frac{a+b}2\biggr) -\frac1{b-a}\int_a^bf(x)\td x\biggr|\\
&\le\frac{b-a}{2^{s/q+2}}\biggl[\frac{1}{(s+1)(s+2)}\biggr]^{1/q} \biggl(\frac12-\lambda+\lambda^2\biggr)^{1-1/q} \bigl\{\bigl[\bigl(2(2-\lambda)^{s+2}\\
&\quad+((s+2)\lambda-2)2^{s+1}+(s+2)\lambda-s-3\bigr)|f'(a)|^q+\bigl(2\lambda^{s+2}+s+1\\
&\quad-(s+2)\lambda\bigr)|f'(b)|^{q}\bigr]^{1/q}
+\bigl[\bigl(2\lambda^{s+2}-(s+2)\lambda+s+1\bigr)|f'(a)|^q \\
&\quad+\bigl(2(2-\lambda)^{s+2}+((s+2)\lambda-2)2^{s+1} +(s+2)\lambda-s-3\bigr)|f'(b)|^{q}\bigr]^{1/q}\bigr\};
\end{align*}
\item
when $\lambda=\mu$, $-1<s\le 1$, and $q=1$,
\begin{multline}\label{sep-eq-3.6}
\biggl\lvert\lambda\frac{f(a)+f(b)}2+(1-\lambda)f\biggl(\frac{a+b}2\biggr)
 -\frac1{b-a}\int_a^bf(x)\td x\biggr|\\
 \le\frac{(b-a)\bigl\{(2-\lambda)^{s+2} +\lambda^{s+2}+[(s+2)\lambda-2]2^{s}-1\bigr\}\bigl(|f'(a)|+|f'(b)|\bigr)}{(s+1)(s+2)2^{s+1}};
\end{multline}
\item
when $\lambda=\mu$, $-1<s\le 1$, and $\lambda=\mu=1$, we have
\begin{equation}\label{sep-eq-3.7}
\begin{split}
&\quad\biggl\lvert\frac{f(a)+f(b)}2 -\frac1{b-a}\int_a^bf(x)\td x\biggr|\\
&\le\frac{b-a}{8}\biggl[\frac{2}{(s+1)(s+2)}\biggr]^{1/q} \biggl\{\biggl[\biggl(2s+\frac1{2^s}\biggr)|f'(a)|^q +\frac{|f'(b)|^{q}}{2^s}\biggr]^{1/q}\\
&\quad+\biggl[\biggl(2s+\frac1{2^s}\biggr)|f'(b)|^{q}+\frac{|f'(a)|^q}{2^s}\biggr]^{1/q}\biggr\}\\
&\le\frac{(b-a)\bigl[|f'(a)|^{q}+|f'(b)|^{q}\bigr]^{1/q}}{4} \biggl[\frac{4+(1/2)^{s-1}}{(s+1)(s+2)}\biggr]^{1/q}.
\end{split}
\end{equation}
\end{enumerate}
\end{cor}

\begin{rem}\label{rem-cor-1}
The inequality~\eqref{eq-xi-1.9} is a special case of~\eqref{sep-eq-3.7} applied to $0<s\le1$.
The inequality~\eqref{sep-eq-1.12} can be deduced from~\eqref{sep-eq-3.6} applied to $\lambda=\mu=\frac{1}{3}$ and $0<s\le1$. These show that Theorem~\ref{thm1-September-2011-xi} and its corollaries generalize some main results obtained in~\cite{Kirmaci--Bakula-Ozdemir-Pecaric-AMC-26-35, Sarikaya-Set-Ozdemir-CMA-2191-2199}.
\end{rem}

\begin{cor}\label{cor-3.1-3-2011-xi}
Under conditions of Theorem~\ref{thm1-September-2011-xi},
\begin{enumerate}
\item
when $s=1$, we have
\begin{multline*}
\biggl\lvert\frac{\lambda f(a)+\mu f(b)}2+\frac{2-\lambda-\mu}2f\biggl(\frac{a+b}2\biggr) -\frac1{b-a}\int_a^bf(x)\td x\biggr|\le\frac{b-a}{2^{1/q+2}}\biggl(\frac{1}{6}\biggl)^{1/q} \\
\times\biggl\{\biggl(\frac12-\lambda+\lambda^2\biggr)^{1-1/q} \bigl[\bigl(4-9\lambda+12\lambda^2-2\lambda^3\bigr)|f'(a)|^q +\bigl(2-3\lambda+2\lambda^3\bigr)|f'(b)|^{q}\bigr]^{1/q}\\
+\biggl(\frac12-\mu+\mu^2\biggr)^{1-1/q} \bigl[(2-3\mu+2\mu^3\bigr)|f'(a)|^q +(4-9\mu+12\mu^2-2\mu^3)|f'(b)|^{q}\bigr]^{1/q}\biggr\};
\end{multline*}
\item
when $s=1$ and $q=1$, we have
\begin{multline*}
\biggl\lvert\frac{\lambda f(a)+\mu f(b)}2+\frac{2-\lambda-\mu}2f\biggl(\frac{a+b}2\biggr) -\frac1{b-a}\int_a^bf(x)\td x\biggr|\\\le\frac{b-a}{48}
\bigl\{\bigl(6-9\lambda+12\lambda^2-2\lambda^3-3\mu+2\mu^3\bigr)|f'(a)| +\bigl(6+3\lambda+2\lambda^3-9\mu+12\mu^2-2\mu^3\bigr)|f'(b)|\bigr\};
\end{multline*}
\item
when $s=1$ and $\lambda=\mu$,
\begin{multline*}
\biggl|\frac{\lambda[f(a)+f(b)]}2+(1-\lambda)f\biggl(\frac{a+b}2\biggr) -\frac1{b-a}\int_a^bf(x)\td x\biggr|
\le\frac{b-a}{4}\biggl(\frac{1}{12}\biggl)^{1/q}\biggl(\frac12-\lambda+\lambda^2\biggr)^{1-1/q}\\
\times\bigl\{ \bigl[\bigl(4-9\lambda+12\lambda^2-2\lambda^3\bigr)|f'(a)|^q
+\bigl(2-3\lambda+2\lambda^3\bigr)|f'(b)|^{q}\bigr]^{1/q}\\
+\bigl[(2-3\lambda+2\lambda^3\bigr)|f'(a)|^q +\bigl(4-9\lambda+12\lambda^2-2\lambda^3\bigr)|f'(b)|^{q}\bigr]^{1/q}\bigr\};
\end{multline*}
\item
when $s=1$, $q=1$, and $\lambda=\mu$, we have
\begin{multline}\label{sep-eq-3.11}
\biggl|\frac{\lambda[f(a)+f(b)]}2+(1-\lambda)f\biggl(\frac{a+b}2\biggr) -\frac1{b-a}\int_a^bf(x)\td x\biggr|\\
\le\frac{b-a}{8}\bigl(1-2\lambda+2\lambda^2\bigr)\bigl[|f'(a)|+|f'(b)|\bigr].
\end{multline}
\end{enumerate}
\end{cor}

\begin{rem}\label{rem-cor-2}
Letting $\lambda=1$ in ~\eqref{sep-eq-3.11} yields the inequality~\eqref{eq-xi-1.5} in~\cite{Dragomir-Agarwal-AML-98-95}.
\end{rem}

\begin{cor}\label{cor-3.1-5-2011-xi}
Let $f:I\subseteq\mathbb{R}\to\mathbb{R}$ be differentiable on $I^\circ$, $a,b\in I$ with $a<b$, and $f'\in L[a,b]$. If $|f'(x)|^q$ is convex on $[a,b]$ for $q\ge1$, then
\begin{align*}
\begin{split}
&\quad\biggl\lvert\frac12\biggl[\frac{f(a)+f(b)}2+f\biggl(\frac{a+b}2\biggr)\biggr]-\frac1{b-a}\int_a^bf(x)\td x\biggr|\\
&\le\frac{b-a}{16}\biggl\{\biggl[\frac{3|f'(a)|^q+|f'(b)|^q}{4}\biggr]^{1/q} +\biggl[\frac{|f'(a)|^q+3|f'(b)|^q}{4}\biggr]^{1/q}\biggr\},
\end{split}\\
\begin{split}
&\quad\biggl\lvert{\frac13\biggl[f(a)+f(b)}+f\biggl(\frac{a+b}2\biggr)\biggr]-\frac1{b-a}\int_a^bf(x)\td x\biggr|\\*
&\le\frac{5(b-a)}{72} \biggl\{\biggl[\frac{37|f'(a)|^q+8|f'(b)|^q}{45}\biggr]^{1/q} +\biggl[\frac{8|f'(a)|^q+37|f'(b)|^q}{45}\biggr]^{1/q}\biggr\},
\end{split}\\
\begin{split}
&\quad\biggl\lvert\frac16\biggl[{f(a)+f(b)}+4f\biggl(\frac{a+b}2\biggr)\biggr]-\frac1{b-a}\int_a^bf(x)\td x\biggr|\\*
&\le\frac{5(b-a)}{72}\biggl\{\biggl[\frac{61|f'(a)|^q+29|f'(b)|^q}{90}\biggr]^{1/q} +\biggl[\frac{29|f'(a)|^q+61|f'(b)|^q}{90}\biggr]^{1/q}\biggr\}.
\end{split}
\end{align*}
\end{cor}

\begin{thm}\label{thm2-September-2011-xi}
Let $f:I\subseteq\mathbb{R}_0\to\mathbb{R}$ be differentiable on $I^\circ$, $a,b\in I$ with $a<b$, and $f'\in L[a,b]$. If $|f'(x)|^q$ for $q\ge1$ is an extended $s$-convex function on $[a,b]$, then for $s\in(-1, 1]$ and $0\le\lambda,\mu\le 1$,
\begin{align*}
&\quad\biggl\lvert\frac{\lambda f(a)+\mu f(b)}2+\frac{2-\lambda-\mu}2f\biggl(\frac{a+b}2\biggr) -\frac1{b-a}\int_a^bf(x)\td x\biggr|\\
&\le\frac{b-a}{4}\biggl[\frac{1}{(s+1)(s+2)}\biggr]^{1/q}\biggl\{\biggl(\frac12-\lambda+\lambda^2\biggr)^{1-1/q} \biggl[\bigl(2(1-\lambda)^{s+2}+(s+2)\lambda-1\bigr)|f'(a)|^q\\
&\quad+\bigl(2\lambda^{s+2}+s+1-(s+2)\lambda\bigr) \biggl|f'\biggl(\frac{a+b}{2}\biggr)\biggr|^{q}\biggr]^{1/q}
+\biggl(\frac12-\mu+\mu^2\biggr)^{1-1/q}
\biggl[\bigl(2\mu^{s+2}+s+1\\
&\quad-(s+2)\mu\bigr)\biggl|f'\biggl(\frac{a+b}{2}\biggr)\biggr|^q
+\bigl(2(1-\mu)^{s+2}+(s+2)\mu-1\bigr)|f'(b)|^{q}\biggr]^{1/q}\biggr\}\\
&\le\frac{b-a}{2^{s/q+2}}\biggl[\frac{1}{(s+1)(s+2)}\biggr]^{1/q} \biggl\{\biggl(\frac12-\lambda+\lambda^2\biggr)^{1-1/q} \bigl[\bigl((1-\lambda)^{s+2}2^{s+1}+2\lambda^{s+2}\\
&\quad+s+1+((s+2)\lambda-1)2^{s}-(s+2)\lambda\bigr)|f'(a)|^{q}+
\bigl(2\lambda^{s+2}-(s+2)\lambda+s+1\bigr)|f'(b)|^{q}\bigr]^{1/q}\\
&\quad
+\biggl(\frac12-\mu+\mu^2\biggr)^{1-1/q}
\bigl[\bigl(2\mu^{s+2}-(s+2)\mu+s+1\bigr)|f'(a)|^q\\
&\quad +\bigl((1-\mu)^{s+2}2^{s+1}+2\mu^{s+1}
+\bigl((s+2)\mu-1\bigr)2^{s}-(s+2)\mu+s+1\Bigr)|f'(b)|^{q}\bigr]^{1/q}\biggr\}.
\end{align*}
\end{thm}

\begin{proof}
By similar arguments as in the proof of Theorem~\ref{thm1-September-2011-xi} and by the extended $s$-convexity of the function $|f'(x)|^q$, we have
\begin{align*}
&\quad\biggl|\frac{\lambda f(a)+\mu f(b)}2 +\frac{2-\lambda-\mu}2f\biggl(\frac{a+b}2\biggr) -\frac1{b-a}\int_a^bf(x)\td x\biggr|\\
&\le\frac{b-a}4\biggl\{\biggl(\int_0^1|1-\lambda-t|\td t\biggr)^{1-1/q} \biggl[\int_0^1|1-\lambda-t|\biggl((1-t)^s\biggl|f'\biggl(\frac{a+b}2\biggr)\biggr|^q+t^s|f'(a)|^q\biggr)\td t\biggr]^{1/q}\\
&\quad
+\biggl(\int_0^1|\mu-t|\td t\biggr)^{1-1/q} \biggl[\int_0^1|\mu-t| \biggl(t^s\biggl|f'\biggl(\frac{a+b}{2}\biggr)\biggr|^q+(1-t)^s|f'(b)|^q\biggr)\td t \biggr]^{1/q}\biggr\}\\
&=\frac{b-a}4\biggl\{\biggl(\frac{1}{2}-\lambda+\lambda^2\biggr)^{1-1/q} \biggl[\frac{1}{(s+1)(s+2)}\biggl(\bigl(2(1-\lambda)^{s+2}+(s+2)\lambda-1\bigr)|f'(a)|^q\\
&\quad
+\bigl(2\lambda^{s+2}-(s+2)\lambda+s+1\bigr) \biggl|f'\biggl(\frac{a+b}{2}\biggr)\biggr|^{q}\biggr)\biggr]^{1/q}
+\biggl(\frac{1}{2}-\mu+\mu^2\biggr)^{1-1/q} \biggl[\frac{1}{(s+1)(s+2)}\\
&\quad\times\biggl(\bigl(2\mu^{s+2}+s+1-(s+2)\mu\bigr) \biggl|f'\biggl(\frac{a+b}{2}\biggr)\biggr|^{q}
+\bigl(2(1-\mu)^{s+2}+(s+2)\mu-1\bigr)|f'(a)|^q\biggr)\biggr]^{1/q}\biggr\}.
\end{align*}
Combining this with
\begin{equation*}
\biggr|f'\biggl(\frac{a+b}{2}\biggr)\biggl|^{q}\le \biggl(\frac{1}{2}\biggr)^s\bigl[|f'(a)\bigr|^{q}+|f'(b)|^q\bigr]
\end{equation*}
leads to Theorem~\ref{thm2-September-2011-xi}.
\end{proof}

\begin{cor}\label{cor-3.2-1-2011-xi}
Under conditions of Theorem~\ref{thm2-September-2011-xi}, when $q=1$, we have
\begin{align*}
&\quad\biggl\lvert\frac{\lambda f(a)+\mu f(b)}2+\frac{2-\lambda-\mu}2f\biggl(\frac{a+b}2\biggr) -\frac1{b-a}\int_a^bf(x)\td x\biggr|\\
&\le\frac{b-a}{4(s+1)(s+2)}\biggl\{\bigl[2(1-\lambda)^{s+2}+(s+2)\lambda-1\bigr]|f'(a)| +\bigl[2\lambda^{s+2}+2\mu^{s+2}\\*
&\quad+(s+2)(1-\lambda-\mu)+s\bigl] \biggl|f'\biggl(\frac{a+b}{2}\biggr)\biggr|+\bigl[2(1-\mu)^{s+2}+(s+2)\mu-1\bigr]|f'(b)|\biggr\}\\
&\le\frac{b-a}{2^{s+2}(s+1)(s+2)}\bigl\{\bigl[(1-\lambda)^{s+2}2^{s+1}+2\lambda^{s+2} +2\mu^{s+2}+((s+2)\lambda-1)2^{s}\\
&\quad+(s+2)(1-\lambda-\mu)+s\bigl]|f'(a)| +\bigl[2\lambda^{s+2}+(1-\mu)^{s+2}2^{s+1}\\*
&\quad+2\mu^{s+2}+(s+2)(1-\lambda-\mu) +((s+2)\mu-1)2^s+s\bigr]|f'(b)|\bigr\}.
\end{align*}
\end{cor}

\begin{cor}\label{cor-3.2-2-2011-xi}
Under conditions of Theorem~\ref{thm2-September-2011-xi}, if $\lambda=\mu$, then
\begin{align*}
&\quad\biggl\lvert\frac{\lambda[f(a)+f(b)] }2+(1-\lambda)2f\biggl(\frac{a+b}2\biggr) -\frac1{b-a}\int_a^bf(x)\td x\biggr|\\
&\le\frac{b-a}{4}\biggl[\frac{1}{(s+1)(s+2)}\biggr]^{1/q}\biggl(\frac12-\lambda+\lambda^2\biggr)^{1-1/q}
\biggl\{\biggl[\bigl(2(1-\lambda)^{s+2}+(s+2)\lambda-1\bigr)|f'(a)|^q\\
&\quad+\bigl(2\lambda^{s+2}+s+1-(s+2)\lambda\bigr)
\biggl|f'\biggl(\frac{a+b}{2}\biggr)\biggr|^{q}\biggr]^{1/q}
+\biggl[\bigl(2\mu^{s+2}+s+1\\
&\quad-(s+2)\lambda\bigr)\biggl|f'\biggl(\frac{a+b}{2}\biggr)\biggr|^q
+\bigl(2(1-\lambda)^{s+2}+(s+2)\lambda-1\bigr)|f'(b)|^{q}\biggr]^{1/q}\biggr\}\\
&\le\biggl(\frac12-\lambda+\lambda^2\biggr)^{1-1/q} \frac{b-a}{2^{s/q+2}} \biggl[\frac{1}{(s+1)(s+2)}\biggr]^{1/q} \bigl\{\bigl[\bigl(2^{s+1}(1-\lambda)^{s+2}+2\lambda^{s+2}-(s+2)\lambda\\
&\quad+((s+2)\lambda-1)2^{s}+s+1\bigr)|f'(a)|^{q}+
\bigl(2\lambda^{s+2}-(s+2)\lambda+s+1\bigr)|f'(b)|^{q}\bigr]^{1/q}\\
&\quad
+\bigl[\bigl(2\lambda^{s+2}-(s+2)\lambda+s+1\bigr)|f'(a)|^q
+\bigl(2\lambda^{s+1}+(1-\lambda)^{s+2}2^{s+1}
+((s+2)\lambda-1)2^{s}\\
&\quad
-(s+2)\lambda+s+1\bigr)|f'(b)|^{q}\bigr]^{1/q}\bigr\}.
\end{align*}
\end{cor}

\begin{rem}\label{rem-cor-2-2}
The inequality~\eqref{eq-xi-1.11} can be deduced from letting $\lambda=\mu=0$ in Corollary~\ref{cor-3.2-2-2011-xi}.
\end{rem}

\begin{cor}\label{cor-3.2-3-2011-xi}
Under conditions of Theorem~\ref{thm2-September-2011-xi}, when $s=1$, we have
\begin{align*}
&\quad\biggl\lvert\frac{\lambda f(a)+\mu f(b)}2+\frac{2-\lambda-\mu}2f\biggl(\frac{a+b}2\biggr) -\frac1{b-a}\int_a^bf(x)\td x\biggr|\\
&\le\frac{b-a}{4}\biggl(\frac{1}{6}\biggr)^{1/q}\biggl\{\biggl(\frac12-\lambda+\lambda^2\biggr)^{1-1/q} \biggl[\bigl(1-3\lambda+6 \lambda^2-2\lambda^3\bigr)|f'(a)|^q\\
&\quad
+\bigl(2\lambda^{3}-3\lambda+3\bigr)\biggl|f'\biggl(\frac{a+b}{2}\biggr)\biggr|^{q}\biggr]^{1/q}
+\biggl(\frac12-\mu+\mu^2\biggr)^{1-1/q}
\biggl[\bigl(2\mu^3-3\mu+2\bigr)\biggl|f'\biggl(\frac{a+b}{2}\biggr)\biggr|^q\\
&\quad+\bigl(1-3\mu+6\mu^2-2\mu^3\bigr)|f'(b)|^{q}\biggr]^{1/q}\biggr\}\\
&\le\frac{b-a}{2^{1/q+2}}\biggl(\frac{1}{6}\biggr)^{1/q} \biggl\{\biggl(\frac12-\lambda+\lambda^2\biggr)^{1-1/q} \bigl[\bigl(4-9 \lambda+12 \lambda^2-2 \lambda^3\bigr)|f'(a)|^{q}\\
&\quad+\bigl(2\lambda^{3}-3\lambda+2\bigr)|f'(b)|^{q}\bigr]^{1/q}
+\biggl(\frac12-\mu+\mu^2\biggr)^{1-1/q}\bigl[\bigl(2\mu^{3}-3\mu+2\bigr)|f'(a)|^q\\
&\quad
+\bigl(4 - 9\mu + 12 \mu^2 - 2 \mu^3\bigr)|f'(b)|^{q}\bigr]^{1/q}\biggr\}.
\end{align*}
\end{cor}

\begin{thm}\label{thm3-September-2011-xi}
Let $f:I\subseteq\mathbb{R}_0\to\mathbb{R}$ be differentiable on $I^\circ$, $a,b\in I$ with $a<b$, and $f'\in L[a,b]$. If $|f'(x)|^q$ for $q\ge1$ is an extended $s$-convex function on $[a,b]$, then for  $s\in(-1, 1]$ and $0\le\lambda,\mu\le 1$,
\begin{enumerate}
\item
when $q=1$, we have
\begin{multline*}
\biggl|\frac{\lambda f(a)+\mu f(b)}2 +\frac{2-\lambda-\mu}2f\biggl(\frac{a+b}2\biggr) -\frac1{b-a}\int_a^bf(x)\td x\biggr|
\le\frac{b-a}{2^{s+2}(s+1)}\\
\times\biggl\{\biggl(\frac12-\lambda+\lambda^2\biggr) \bigl[|f'(a)|+(2^{s+1}-1)|f'(b)|\bigr]
+\biggl(\frac12-\mu+\mu^2\biggr) \bigl[(2^{s+1}-1)|f'(a)|+|f'(b)|\bigr]\biggr\};
\end{multline*}
\item
when $q>1$, we have
\begin{multline}\label{thm3-2-September-2011-xi-ineq}
\biggl|\frac{\lambda f(a)+\mu f(b)}2 +\frac{2-\lambda-\mu}2f\biggl(\frac{a+b}2\biggr) -\frac1{b-a}\int_a^bf(x)\td x\biggr|
\le\frac{b-a}{2^{s/q+2}}\biggl(\frac{q-1}{2q-1}\biggr)^{1-1/q}\\
\times\biggl(\frac{1}{s+1}\biggr)^{1/q}\bigl\{\bigr[(1-\lambda)^{(2q-1)/(q-1)}
+\lambda^{(2q-1)/(q-1)}\bigl]^{1-1/q} \bigl[(2^{s+1}-1)|f'(a)|^q+|f'(b)|^q\bigr]^{1/q}\\
+\bigr[\mu^{(2q-1)/(q-1)}
+(1-\mu)^{(2q-1)/(q-1)}\bigl]^{1-1/q}
\bigl[|f'(a)|^q+(2^{s+1}-1)|f'(b)|^q\bigr]^{1/q}\bigr\}.
\end{multline}
\end{enumerate}
\end{thm}

\begin{proof}
For $q>1$, by the extended $s$-convexity of $|f'(x)|^q$ on $[a,b]$, Lemma~\ref{lem1-September-2011-xi}, and H\"older integral inequality, we have
\begin{align*}
&\quad\biggl|\frac{\lambda f(a)+\mu f(b)}2 +\frac{2-\lambda-\mu}2f\biggl(\frac{a+b}2\biggr)
-\frac1{b-a}\int_a^bf(x)\td x\biggr|\\
&\le\frac{b-a}4\biggl[\int_0^1|1-\lambda-t|\biggl|f'\biggl(ta+(1-t)\frac{a+b}2\biggr)\biggr|\td t
+\int_0^1|\mu-t|\biggl|f'\biggl(t\frac{a+b}2
+(1-t)b\biggr)\biggr|\td t\biggr]\\
&\le\frac{b-a}{2^{s/q+2}}\biggl\{\biggl(\int_0^1|1-\lambda-t|^{q/(q-1)}\td t\biggr)^{1-1/q} \biggl[\int_0^1\bigl((1+t)^s|f'(a)|^q+(1-t)^s|f'(b)|^q\bigr)\td t\biggr]^{1/q} \\
&\quad+\biggl(\int_0^1|\mu-t|^{q/(q-1)}\td t\biggr)^{1-1/q}
\biggl[\int_0^1\bigl(t^s|f'(a)|^q+(2-t)^s|f'(b)|^q\bigr)\td t\biggr]^{1/q}\biggr\}.
\end{align*}
In virtue of Lemma~\ref{lem2-September-2011-xi}, a direct calculation yields
\begin{align*}
\int_0^1|1-\lambda-t|^{q/(q-1)}\td t &=\frac{q-1}{2q-1}\bigr[\lambda^{(2q-1)/(q-1)}+(1-\lambda)^{(2q-1)/(q-1)}\bigl],\\
\int_0^1|\mu-t|^{q/(q-1)}\td t
&=\frac{q-1}{2q-1} \bigr[\mu^{(2q-1)/(q-1)}+(1-\mu)^{(2q-1)/(q-1)}\bigl].
\end{align*}
A straightforward computation gives
\begin{align*}
\int_0^1\bigl[(1+t)^s|f'(a)|^q+(1-t)^s|f'(b)|^q\bigr]\td t &=\frac{(2^{s+1}-1)|f'(a)|^q+|f'(b)|^q}{s+1},\\
\int_0^1\bigl[t^s|f'(a)|^q+(2-t)^s|f'(b)|^q\bigr]\td t &=\frac{|f'(a)|^q+(2^{s+1}-1)|f'(b)|^q}{s+1}.
\end{align*}
Substituting the last four equalities into the first inequality and simplifying establish the inequality~\eqref{thm3-2-September-2011-xi-ineq}.
\par
For $q=1$, utilizing the extended $s$-convexity of $|f'(x)|^q$ on $[a,b]$, Lemma~\ref{lem1-September-2011-xi}, and H\"older integral inequality, we have
\begin{align*}
&\quad\biggl|\frac{\lambda f(a)+\mu f(b)}2 +\frac{2-\lambda-\mu}2f\biggl(\frac{a+b}2\biggr) -\frac1{b-a}\int_a^bf(x)\td x\biggr|\\
&\le\frac{b-a}4\biggl[\int_0^1|1-\lambda-t|\biggl|f'\biggl(ta+(1-t)\frac{a+b}2\biggr)\biggr|\td t +\int_0^1|\mu-t|\biggl|f'\biggl(t\frac{a+b}2+(1-t)b\biggr)\biggr|\td t\biggr]\\
&\le\frac{b-a}{2^{s+2}}\biggl\{\biggl(\int_0^1|1-\lambda-t|\td t\biggr) \int_0^1\bigl[(1+t)^s|f'(a)|+(1-t)^s|f'(b)|\bigr]\td t\\
&\quad+\biggl(\int_0^1|\mu-t|\td t\biggr) \int_0^1\bigl[t^s|f'(a)|+(2-t)^s|f'(b)|\bigr]\td t\biggr\}\\
&=\frac{b-a}{2^{s+2}(s+1)}\biggl\{\biggl(\frac12-\lambda+\lambda^2\biggr) \bigl[|f'(a)|+(2^{s+1}-1)|f'(b)|\bigr] \\
&\quad+\biggl(\frac12-\mu+\mu^2\biggr)\bigl[(2^{s+1}-1)|f'(a)|+|f'(b)|\bigr]\biggr\}.
\end{align*}
Theorem~\ref{thm3-September-2011-xi} is thus proved.
\end{proof}

\begin{cor}\label{cor-3.3-1-2011-xi}
Under conditions of Theorem~\ref{thm3-September-2011-xi},
\begin{enumerate}
\item
when $\lambda=\mu$ and $q>1$, we have
\begin{align*}
&\quad\biggl|\frac{\lambda [f(a)+f(b)]}2 +(1-\lambda)f\biggl(\frac{a+b}2\biggr) -\frac1{b-a}\int_a^bf(x)\td x\biggr|\\
&\le\frac{b-a}{2^{s/q+2}}\biggl(\frac{q-1}{2q-1}\biggr)^{1-1/q}\biggl(\frac{1}{s+1}\biggr)^{1/q} \bigr[\lambda^{(2q-1)/(q-1)}+(1-\lambda)^{(2q-1)/(q-1)}\bigl]^{1-1/q} \\ &\quad\times\bigl\{\bigl[(2^{s+1}-1)|f'(a)|^q+|f'(b)|^q\bigr]^{1/q}
+\bigl[|f'(a)|^q+(2^{s+1}-1)|f'(b)|^q\bigr]^{1/q}\bigr\};
\end{align*}
\item
when $\lambda=\mu=0,1$ and $q>1$, we have
\begin{multline*}
\biggl|f\biggl(\frac{a+b}2\biggr) -\frac1{b-a}\int_a^bf(x)\td x \biggr|\le\frac{b-a}{2^{s/q+2}}\biggl(\frac{q-1}{2q-1}\biggr)^{1-1/q} \biggl(\frac{1}{s+1}\biggr)^{1/q}\\
\times\bigl\{\bigl[(2^{s+1}-1)|f'(a)|^q+|f'(b)|^q\bigr]^{1/q}
+\bigl[|f'(a)|^q+(2^{s+1}-1)|f'(b)|^q\bigr]^{1/q}\bigr\}
\end{multline*}
and
\begin{multline*}
\biggl|\frac{f(a)+f(b)}2 -\frac1{b-a}\int_a^bf(x)\td x\biggr|
\le\frac{b-a}{2^{s/q+2}}\biggl(\frac{q-1}{2q-1}\biggr)^{1-1/q}\biggl(\frac{1}{s+1}\biggr)^{1/q} \\ \times\bigl\{\bigl[(2^{s+1}-1)|f'(a)|^q+|f'(b)|^q\bigr]^{1/q}
+\bigl[|f'(a)|^q+(2^{s+1}-1)|f'(b)|^q\bigr]^{1/q}\bigr\}.
\end{multline*}
\end{enumerate}
\end{cor}

\begin{cor}\label{cor-3.3-2-2011-xi}
Under conditions of Theorem~\ref{thm3-September-2011-xi},
\begin{enumerate}
\item
when $q=1$ and $s=1$, we have
\begin{multline*}
\biggl|\frac{\lambda f(a)+\mu f(b)}2 +\frac{2-\lambda-\mu}2f\biggl(\frac{a+b}2\biggr) -\frac1{b-a}\int_a^bf(x)\td x\biggr|\\
\le\frac{b-a}{16}\biggl\{\biggl(\frac12-\lambda+\lambda^2\biggr) \bigl[|f'(a)|+3|f'(b)|\bigr]+\biggl(\frac12-\mu+\mu^2\biggr) \bigl[3|f'(a)|+|f'(b)|\bigr]\biggr\};
\end{multline*}
\item
when $q>1$ and $s=1$, we have
\begin{multline*}
\biggl|\frac{\lambda f(a)+\mu f(b)}2 +\frac{2-\lambda-\mu}2f\biggl(\frac{a+b}2\biggr) -\frac1{b-a}\int_a^bf(x)\td x\biggr|
\le\frac{b-a}{2^{2/q+2}}\biggl(\frac{q-1}{2q-1}\biggr)^{1-1/q} \\
\times\bigl\{\bigr[\lambda^{(2q-1)/(q-1)}+(1-\lambda)^{(2q-1)/(q-1)}\bigl]^{1-1/q} \bigl[3|f'(a)|^q+|f'(b)|^q\bigr]^{1/q}\\
+\bigr[\mu^{(2q-1)/(q-1)}+(1-\mu)^{(2q-1)/(q-1)}\bigl]^{1-1/q}
\bigl[|f'(a)|^q+3|f'(b)|^q\bigr]^{1/q}\bigr\};
\end{multline*}
\item
when $q=1$, $\lambda=\mu$, and $s=1$, we have
\begin{multline}\label{sep-eq-3.30}
\biggl|\frac{\lambda [f(a)+ f(b)]}2 +(1-\lambda)f\biggl(\frac{a+b}2\biggr) -\frac1{b-a}\int_a^bf(x)\td x\biggr|\\*
\le\frac{b-a}{4}\bigl\{\bigl(1-2\lambda+2\lambda^2\bigr) \bigl[|f'(a)|+|f'(b)|\bigr]\bigr\};
\end{multline}
\item
when $q>1$, $\lambda=\mu$, and $s=1$, we have
\begin{multline}\label{sep-eq-3.31}
\biggl|\frac{\lambda [f(a)+ f(b)]}2 +(1-\lambda)f\biggl(\frac{a+b}2\biggr) -\frac1{b-a}\int_a^bf(x)\td x\biggr|\le\biggl(\frac{q-1}{2q-1}\biggr)^{1-1/q} \frac{b-a}{2^{2/q}}\\
\times\bigr[\lambda^{(2q-1)/(q-1)}+(1-\lambda)^{(2q-1)/(q-1)}\bigl]^{1-1/q} \bigl[|f'(a)|^q+|f'(b)|^q\bigr]^{1/q}.
\end{multline}
\end{enumerate}
\end{cor}

\begin{thm}\label{thm4-September-2011-xi}
Let $f:I\subseteq\mathbb{R}_0\to\mathbb{R}$ be differentiable on $I^\circ$, $a,b\in I$ with $a<b$, and $f'\in L[a,b]$. If $|f'(x)|^q$ for $q\ge1$ is an extended $s$-convex function on $[a,b]$, then, for $s\in(-1, 1]$ and $0\le\lambda,\mu\le 1$,
\begin{enumerate}
\item
when $q=1$, we have
\begin{align*}
&\quad\biggl\lvert\frac{\lambda f(a)+\mu f(b)}2+\frac{2-\lambda-\mu}2f\biggl(\frac{a+b}2\biggr) -\frac1{b-a}\int_a^bf(x)\td x\biggr|\\
&\le\frac{b-a}{4(s+1)}\biggl\{\biggl(\frac12-\lambda+\lambda^2\biggr) \biggl[|f'(a)|+\biggr|f'\biggl(\frac{a+b}2\biggr)\biggr|\biggr]\\
&\quad+\biggl(\frac12-\mu+\mu^2\biggr)
\biggl[\biggr|f'\biggl(\frac{a+b}2\biggr)\biggr|+|f'(b)|\biggr]\biggr\}\\
&\le\frac{b-a}{2^{s+2}(s+1)}\biggl\{\biggl(\frac12-\lambda+\lambda^2\biggr)[(2^s+1)|f'(a)|+|f'(b)|]\\
&\quad+\biggl(\frac12-\mu+\mu^2\biggr)
[|f'(a)|+(2^s+1)|f'(b)|]\biggr\};
\end{align*}
\item
when $q>1$, we have
\begin{align*}
&\quad\biggl\lvert\frac{\lambda f(a)+\mu f(b)}2+\frac{2-\lambda-\mu}2f\biggl(\frac{a+b}2\biggr) -\frac1{b-a}\int_a^bf(x)\td x\biggr|\\
&\le\frac{b-a}{4}\biggl(\frac{q-1}{2q-1}\biggr)^{1-1/q} \biggl(\frac{1}{s+1}\biggr)^{1/q} \biggl\{\bigr[\lambda^{(2q-1)/(q-1)}+(1-\lambda)^{(2q-1)/(q-1)}\bigl]^{1-1/q} \\
&\quad\times\biggl[|f'(a)|^q+\biggr|f'\biggl(\frac{a+b}2\biggr)\biggr|^q\biggr]^{1/q}
+\bigr[\mu^{(2q-1)/(q-1)}+(1-\mu)^{(2q-1)/(q-1)}\bigl]^{1-1/q}\\
&\quad\times\biggl[\biggr|f'\biggl(\frac{a+b}2\biggr)\biggr|^q+|f'(b)|^q\biggr]^{1/q}\biggr\}\\
&\le\frac{b-a}{2^{s/q+2}}\biggl(\frac{q-1}{2q-1}\biggr)^{1-1/q}\biggl(\frac{1}{s+1}\biggr)^{1/q}\\
&\quad\times\bigl\{\bigr[\lambda^{(2q-1)/(q-1)}+(1-\lambda)^{(2q-1)/(q-1)}\bigl]^{1-1/q} \bigl[(2^s+1)|f'(a)|^q+|f'(b)|^q\bigr]^{1/q} \\
&\quad+\bigr[\mu^{(2q-1)/(q-1)}+(1-\mu)^{(2q-1)/(q-1)}\bigl]^{1-1/q}
\bigl[|f'(a)|^q+(2^s+1)|f'(b)|^q\bigr]^{1/q}\bigr\}.
\end{align*}
\end{enumerate}
\end{thm}

\begin{proof}
For $q>1$, since $|f'(x)|^q$ is extended $s$-convex on $[a,b]$, by Lemma~\ref{lem1-September-2011-xi} and H\"older integral inequality, we have
\begin{align*}
&\quad\biggl|\frac{\lambda f(a)+\mu f(b)}2 +\frac{2-\lambda-\mu}2f\biggl(\frac{a+b}2\biggr) -\frac1{b-a}\int_a^bf(x)\td x\biggr|\\
&\le\frac{b-a}4\biggl[\int_0^1|1-\lambda-t|\biggl|f'\biggl(ta+(1-t)\frac{a+b}2\biggr)\biggr|\td t
+\int_0^1|\mu-t|\biggl|f'\biggl(t\frac{a+b}2+(1-t)b\biggr)\biggr|\td t\biggr]\\
&\le\frac{b-a}{4}\biggl\{\biggl(\int_0^1|1-\lambda-t|^{q/(q-1)}\td t\biggr)^{1-1/q} \biggl[\int_0^1\biggl(t^s|f'(a)|^q +(1-t)^s\biggl|f'\biggl(\frac{a+b}2\biggr)\biggr|^q\biggr)\td t\biggr]^{1/q}\\
&\quad+\biggl(\int_0^1|\mu-t|^{q/(q-1)} \td t\biggr)^{1-1/q}\biggl[\int_0^1 \biggr(t^s\biggl|f'\biggl(\frac{a+b}2\biggr)\biggr|^q+(1-t)^s|f'(b)|^q\biggr)\td t\biggr]^{1/q}\biggr\}.
\end{align*}
If $q=1$, we have
\begin{align*}
&\quad\biggl|\frac{\lambda f(a)+\mu f(b)}2 +\frac{2-\lambda-\mu}2f\biggl(\frac{a+b}2\biggr) -\frac1{b-a}\int_a^bf(x)\td x\biggr|\\
&\le\frac{b-a}4\biggl[\int_0^1|1-\lambda-t|\biggl|f'\biggl(ta+(1-t)\frac{a+b}2\biggr)\biggr|\td t
+\int_0^1|\mu-t|\biggl|f'\biggl(t\frac{a+b}2+(1-t)b\biggr)\biggr|\td t\biggr]\\
&\le\frac{b-a}{4}\biggl\{\biggl(\int_0^1|1-\lambda-t|\td t\biggr) \biggl[\int_0^1\biggl(t^s|f'(a)|
+(1-t)^s\biggl|f'\biggl(\frac{a+b}2\biggr)\biggr|\biggr)\td t\biggr]\\
&\quad+\biggl(\int_0^1|\mu-t|\td t\biggr)\biggl[\int_0^1 \biggr(t^s\biggl|f'\biggl(\frac{a+b}2\biggr)\biggr|+(1-t)^s|f'(b)|\biggr)\td t\biggr]\biggr\}.
\end{align*}
Theorem~\ref{thm4-September-2011-xi} is thus proved.
\end{proof}

\begin{cor}\label{cor-3.4-1-2011-xi}
Under conditions of Theorem~\ref{thm4-September-2011-xi},
\begin{enumerate}
\item
when $q=1$ and $\lambda=\mu$, we have
\begin{multline}\label{sep-eq-3.34}
\biggl\lvert\frac{\lambda [f(a)+f(b)]}2+(1-\lambda)f\biggl(\frac{a+b}2\biggr) -\frac1{b-a}\int_a^bf(x)\td x\biggr|\\
\le\frac{b-a}{4(s+1)}\biggl(\frac12-\lambda+\lambda^2\biggr) \biggl[|f'(a)|+2\biggr|f'\biggl(\frac{a+b}2\biggr)\biggr|+|f'(b)|\biggr]\\
\le\frac{b-a}{2^{s+1}(s+1)}\biggl(\frac12-\lambda+\lambda^2\biggr)(2^{s-1}+1)[|f'(a)|+|f'(b)|];
\end{multline}
\item
when $q>1$ and $\lambda=\mu$, we have
\begin{align*}
&\quad\biggl\lvert\frac{\lambda [f(a)+f(b)]}2+(1-\lambda)f\biggl(\frac{a+b}2\biggr)
 -\frac1{b-a}\int_a^bf(x)\td x\biggr|\\
& \le\frac{b-a}{4}\biggl(\frac{q-1}{2q-1}\biggr)^{1-1/q}
 \biggl(\frac{1}{s+1}\biggr)^{1/q} \bigr[\lambda^{(2q-1)/(q-1)}+(1-\lambda)^{(2q-1)/(q-1)}\bigl]^{1-1/q}\\
&\quad\times\biggl\{ \biggl[|f'(a)|^q+\biggr|f'\biggl(\frac{a+b}2\biggr)\biggr|^q\biggr]^{1/q}
+\biggl[\biggr|f'\biggl(\frac{a+b}2\biggr)\biggr|^q+|f'(b)|^q\biggr]^{1/q}\biggr\}\\
&\le\frac{b-a}{2^{s/q+2}}\biggl(\frac{q-1}{2q-1}\biggr)^{1-1/q}\biggl(\frac{1}{s+1}\biggr)^{1/q}
\bigr[\lambda^{(2q-1)/(q-1)}+(1-\lambda)^{(2q-1)/(q-1)}\bigl]^{1-1/q} \\
&\quad\times\bigl\{ \bigl[(2^s+1)|f'(a)|^q+|f'(b)|^q\bigr]^{1/q}
+\bigl[|f'(a)|^q+(2^s+1)|f'(b)|^q\bigr]^{1/q}\bigr\}.
\end{align*}
\end{enumerate}
\end{cor}
\begin{cor}\label{cor-3.4-2-2011-xi}
Under conditions of Theorem~\ref{thm4-September-2011-xi},
\begin{enumerate}
\item
when $q=1$ and $s=1$, we have
\begin{align*}
&\quad\biggl\lvert\frac{\lambda f(a)+\mu f(b)}2+\frac{2-\lambda-\mu}2f\biggl(\frac{a+b}2\biggr) -\frac1{b-a}\int_a^bf(x)\td x\biggr|\\
&\le\frac{b-a}{4(s+1)}\biggl\{\biggl(\frac12-\lambda+\lambda^2\biggr) \biggl[|f'(a)|+\biggr|f'\biggl(\frac{a+b}2\biggr)\biggr|\biggr]
+\biggl(\frac12-\mu+\mu^2\biggr)
\biggl[\biggr|f'\biggl(\frac{a+b}2\biggr)\biggr|+|f'(b)|\biggr]\biggr\}\\
&\le\frac{b-a}{16}\biggl\{\biggl(\frac12-\lambda+\lambda^2\biggr)[3|f'(a)|+|f'(b)|]
+\biggl(\frac12-\mu+\mu^2\biggr)
[|f'(a)|+3|f'(b)|]\biggr\};
\end{align*}
\item
when $q>1$ and $s=1$, we have
\begin{align*}
&\quad\biggl\lvert\frac{\lambda f(a)+\mu f(b)}2+\frac{2-\lambda-\mu}2f\biggl(\frac{a+b}2\biggr) -\frac1{b-a}\int_a^bf(x)\td x\biggr|\\
&\le\frac{b-a}{2^{1/q+2}}\biggl(\frac{q-1}{2q-1}\biggr)^{1-1/q}
\biggl\{\bigr[\lambda^{(2q-1)/(q-1)}+(1-\lambda)^{(2q-1)/(q-1)}\bigl]^{1-1/q} \biggl[\biggr|f'\biggl(\frac{a+b}2\biggr)\biggr|^q+|f'(a)|^q\biggr]^{1/q}\\
&\quad+\bigr[\mu^{(2q-1)/(q-1)}+(1-\mu)^{(2q-1)/(q-1)}\bigl]^{1-1/q}
\biggl[\biggr|f'\biggl(\frac{a+b}2\biggr)\biggr|^q+|f'(b)|^q\biggr]^{1/q}\biggr\}\\
&\le\frac{b-a}{2^{2/q+2}}\biggl(\frac{q-1}{2q-1}\biggr)^{1-1/q}
\bigl\{\bigr[\lambda^{(2q-1)/(q-1)}+(1-\lambda)^{(2q-1)/(q-1)}\bigl]^{1-1/q} \bigl[3|f'(a)|^q
+|f'(b)|^q\bigr]^{1/q} \\
&\quad+\bigr[\mu^{(2q-1)/(q-1)}+(1-\mu)^{(2q-1)/(q-1)}\bigl]^{1-1/q}
\bigl[|f'(a)|^q+3|f'(b)|^q\bigr]^{1/q}\bigr\}.
\end{align*}
\end{enumerate}
\end{cor}

\section{Applications to means}

Finally, we apply some inequalities of Hermite-Hadamard type for extended $s$-convex functions to construct some inequalities for means.
\par
For two positive numbers $a>0$ and $b>0$, let
\begin{equation}
 A(a,b)=\frac{a+b}2\quad \text{and}\quad
  L_s(a,b)=\begin{cases}
    \biggl[\dfrac{b^{s+1}-a^{s+1}}{(s+1)(b-a)}\biggr]^{1/s}, & \text{$s\ne0,-1$ and $a\ne b$,}\\
    \dfrac{b-a}{\ln b-\ln a},&\text{$s=-1$ and $a\ne b$,}\\
    \dfrac1e\biggl(\dfrac{b^b}{a^a}\biggr)^{1/(b-a)}, &\text{$s=0$ and $a\ne b$,}\\
    a,&a=b.
  \end{cases}
\end{equation}
They are called the arithmetic and generalized logarithmic means of two positive numbers $a$ and $b$ respectively.
\par
Let $f(x)=x^s$ for $x>0$, $s>0$, and $q\geq1$.
If $0\le(s-1)q\le1$ and $0\le s-1\le1$, then
\begin{align*}
|f'(t x+(1-t)y)|^q&\le s^q\bigl[t^{(s-1)q} x^{(s-1)q}+(1-t)^{(s-1)q}y^{(s-1)q}\bigr]\\
&\le t^{s-1} |f'(x)|^q+(1-t)^{s-1}|f'(y)|^q
\end{align*}
for $x, y>0$ and $t\in(0,1)$.
If $-1<(s-1)q\le 0$ and $-1<s-1\le 0$, then
\begin{align*}
|f'(tx+(1-t)y)|^q
\le t^{s-1} |f'(x)|^q+(1-t)^{s-1}|f'(y)|^q
\end{align*}
for $x, y>0$ and $t\in(0,1)$.
If $-1<(s-1)q\le 1$ and $-1<s-1\le 1$, then $|f'(x)|^q=|s|^qx^{(s-1)q}$ is an extended $(s-1)$-convex function on $[a,b]$.
\par
Applying Corollary~\ref{cor-3.1-2-2011-xi} to $|s|^qx^{(s-1)q}$ yields the following theorem.

\begin{thm}\label{thm1-Sep-Appl-2011-xi}
Let $b>a>0$, $q\ge1$, $0<s\le 2$, $-1<(s-1)q\le1$, and $0\le\lambda\le1$. Then
\begin{multline}\label{sep-eq-3.32}
\bigl\lvert\lambda A(a^s,b^s)+(1-\lambda)A^s(a,b)-L_s^s(a,b)\bigr|
\le\frac{(b-a)s}{2^{(s-1)/q+2}}\biggl[\frac{1}{s(s+1)}\biggr]^{1/q} \biggl(\frac{1}{2}-\lambda+\lambda^2\biggr)^{1-1/q}\\
\begin{aligned}
&\times\bigl\{\bigl[\bigl(2(2-\lambda)^{s+1}+2^{s}((s+1)\lambda-2)+(s+1)\lambda-s-2\bigr)a^{(s-1)q}\\
&+\bigl(2\lambda^{s+1}+s-(s+1)\lambda\bigr)b^{(s-1)q}\bigr]^{1/q}
+\bigl[\bigl(2\lambda^{s+1}+s-(s+1)\lambda\bigr)a^{(s-1)q}\\
&+\bigl(2(2-\lambda)^{s+1}+2^{s}((s+1)\lambda-2)+(s+1)\lambda-s-2\bigr)b^{(s-1)q}\bigr]^{1/q}\bigr\}.
\end{aligned}
\end{multline}
Specially, if $q=1$, then
\begin{multline}
\bigl\lvert\lambda A(a^s,b^s)+(1-\lambda)A^s(a,b)-L_s^s(a,b)\bigr|\\
\le\frac{(b-a)s}{2^{s-1}s(s+1)}\bigl\{(2-\lambda)^{s+1}+\lambda^{s+1} +[(s+1)\lambda-2]2^{s-1}-1\bigr\}A(a^{s-1},b^{s-1}).
\end{multline}
\end{thm}

Taking $f(x)=x^s$ for $x>0$ and $s>0$ in Corollary~\ref{cor-3.2-2-2011-xi} derives the following inequalities for means.

\begin{thm}\label{thm2-Sep-Appl-2011-xi}
Let $b>a>0$, $q\ge1$, $0<s\le 2$, $-1<(s-1)q\le1$, and $0\le\lambda\le1$. Then
\begin{multline*}
\bigl\lvert\lambda A(a^s, b^s)+(1-\lambda)A^s(a, b) -L_s^s(a,b)\bigr|
\le\frac{(b-a)s}{4}\biggl[\frac{1}{s(s+1)}\biggr]^{1/q} \biggl(\frac12-\lambda+\lambda^2\biggr)^{1-1/q}\\
\begin{aligned}
&\times\bigl\{\bigl[\bigl(2(1-\lambda)^{s+1}+(s+1)\lambda-1\bigr)a^{(s-1)q} +\bigl(2\lambda^{s+2}-(s+1)\lambda+s\bigr)A^{(s-1)q}(a, b)\bigr]^{1/q}\\
&+\bigl[\bigl(2\lambda^{s+1}-(s+1)\lambda+s\bigr)A^{(s-1)q}(a, b) +\bigl(2(1-\lambda)^{s+1}+(s+1)\lambda-1\bigr)b^{(s-1)q}\bigr]^{1/q}\bigr\}.
\end{aligned}
\end{multline*}
In particular, if $q=1$, then
\begin{multline*}
\biggl\lvert\lambda A(a^s, b^s)+(1-\lambda)A^s(a, b) -L_s^s(a,b)\biggr| \\
\le\frac{(b-a)s}{2s(s+1)}\bigl\{\bigl[2(1-\lambda)^{s+1}+(s+1)\lambda-1\bigr]A\bigl(a^{s-1}, b^{s-1}\bigr)
+\bigl[2\lambda^{s+1}+s-(s+1)\lambda\bigl]A^{s-1}(a, b)\bigr\}.
\end{multline*}
\end{thm}

Letting $f(x)=x^s$ for $x>0$ and $s>0$ in Corollary~\ref{cor-3.3-1-2011-xi} generates inequalities below.

\begin{thm}\label{thm3-Sep-Appl-2011-xi}
Let $b>a>0$, $q\ge1$, $0<s\le 2$, and $0\le\lambda\le1$.
\begin{enumerate}
\item
If $q>1$, then
\begin{multline*}
\biggl\lvert\lambda A(a^s, b^s)+(1-\lambda)A^s(a, b) -L_s^s(a,b)\biggr|
\le\frac{(b-a)s}{2^{s/q+2}}\biggl(\frac{q-1}{2q-1}\biggr)^{1-1/q} \biggl(\frac{1}{s+1}\biggr)^{1/q}\bigr[\lambda^{(2q-1)/(q-1)}\\
+(1-\lambda)^{(2q-1)/(q-1)}\bigl]^{1-1/q}\bigl\{ \bigl[(2^s-1)a^{(s-1)q}+b^{(s-1)q}\bigr]^{1/q}
+\bigl[a^{(s-1)q}+(2^s-1)b^{(s-1)q}\bigr]^{1/q}\bigr\}.
\end{multline*}
\item
If $q=1$, then
\begin{equation}\label{thm3-Sep-2011-xi-ineq}
\biggl\lvert\lambda A(a^s, b^s)+(1-\lambda)A^s(a, b) -L_s^s(a,b)\biggr|
\le\frac{(b-a)s}{s+1}\biggl(\frac12-\lambda+\lambda^2\biggr)A\bigl(a^{s-1}, b^{s-1}\bigr).
\end{equation}
\end{enumerate}
\end{thm}

From Corollary~\ref{cor-3.4-1-2011-xi}, it follows that

\begin{thm}\label{thm4-Sep-Appl-2011-xi}
Let $b>a>0$, $q\ge1$, $0<s\le 2$, and $0\le\lambda\le1$.
\begin{enumerate}
\item
If $q>1$ and $-1<(s-1)q\le1$, then
\begin{multline*}
\biggl\lvert\lambda A(a^s, b^s)+(1-\lambda)A^s(a, b) -L_s^s(a,b)\biggr| \le\frac{(b-a)s}{4}\biggl(\frac{q-1}{2q-1}\biggr)^{1-1/q} \biggl(\frac{1}{s+1}\biggr)^{1/q}\bigr[\lambda^{(2q-1)/(q-1)}\\
+(1-\lambda)^{(2q-1)/(q-1)}\bigl]^{1-1/q}\bigl\{ \bigl[a^{(s-1)q}+A^{(s-1)q}(a, b)\bigr]^{1/q}+\bigl[A^{(s-1)q}(a, b)+b^{(s-1)q}\bigr]^{1/q}\bigr\}.
\end{multline*}
\item
If $q=1$, then
\begin{multline}\label{thm4-Sep-2011-xi-ineq-2}
\biggl\lvert\lambda A(a^s, b^s)+(1-\lambda)A^s(a, b) -L_s^s(a,b)\biggr|\\
\le\frac{(b-a)s}{2(s+1)}\biggl(\frac12-\lambda+\lambda^2\biggr)\bigl[A\bigl(a^{s-1}, b^{s-1}\bigr)+A^{s-1}(a, b)\bigr].
\end{multline}
\end{enumerate}
\end{thm}

\subsection*{Acknowledgements}
The authors thank anonymous referees and the Editor, Professor Wataru Takahashi, for their valuable suggestions to and helpful comments on the original version of this paper.
\par
This work was partially supported by the NNSF under Grant No. 11361038 of China and by the Foundation of Research Program of Science and Technology at Universities of Inner Mongolia Autonomous Region under Grant No.~NJZY13159, China.

\end{document}